\documentclass{amsart}
\usepackage{amssymb}
\usepackage{amsfonts}

\newtheorem{theorem}{Theorem}
\theoremstyle{plain}

\newtheorem{definition}{Definition}

\newtheorem{lemma}{Lemma}

\newtheorem{remark}{Remark}

\numberwithin{equation}{section}
\input{tcilatex}

\begin{document}
\title[HERM\.{I}TE-HADAMRD-FEJER TYPE INEQUALITIES FOR FRACTIONAL INTEGRALS]{%
HERMITE-HADAMARD-FEJER TYPE INEQUALITIES FOR $s-$CONVEX \ FUNCTION IN THE
SECOND SENSE VIA FRACTIONAL INTEGRALS}
\author{ERHAN SET$^{\bigstar }$}
\address{$^{\bigstar }$DEPARTMENT OF MATHEMATICS, FACULTY OF ARTS AND
SCIENCES, ORDU UNIVERSITY, 52200, ORDU, TURKEY}
\email{erhanset@yahoo.com}
\author{\.{I}mdat \.{I}scan$^{\blacktriangle }$}
\address{$^{\blacktriangle }$DEPARTMENT OF MATHEMATICS, FACULTY OF ARTS AND
SCIENCES, Giresun UNIVERSITY, Giresun, TURKEY}
\email{imdati@yahoo.com}
\author{Hasan H\"{u}seyin KARA$^{\blacksquare }$}
\address{$^{\blacksquare }$DEPARTMENT OF MATHEMATICS, FACULTY OF ARTS AND
SCIENCES, ORDU UNIVERSITY, 52200, ORDU, TURKEY}
\email{h.huseyin.kara61@gmail.com}

\begin{abstract}
In this paper, we established Hermite-Hadamard-Fejer type inequalities for $%
s-$convex functions in the second sense via fractional integrals. The some
results presented here would provide extansions of those given in earlier
works.
\end{abstract}

\maketitle

\section{\protect\bigskip INTRODUCTION}

The following inequality is well known in the literature as the
Hermite-Hadamard integral inequality\cite{HAD}:%
\begin{equation}
f\left( \frac{a+b}{2}\right) \leq \frac{1}{b-a}\int_{a}^{b}f\left( x\right)
dx\leq \frac{f\left( a\right) +f\left( b\right) }{2}  \label{HH}
\end{equation}%
where $f:I\subset 
\mathbb{R}
\rightarrow 
\mathbb{R}
$ is a convex function on the interval $I$ of real numbers and $a.b\in I$
with $a<b.$ A function $f:\left[ a,b\right] \subset 
\mathbb{R}
\rightarrow 
\mathbb{R}
$ is said to be convex if whenever $x,y\in \left[ a,b\right] $ and $t\in %
\left[ 0,1\right] $ the following inequality holds%
\begin{equation}
f\left( tx+\left( 1-t\right) y\right) \leq tf\left( x\right) +\left(
1-t\right) f\left( y\right) .  \label{convex tanim}
\end{equation}

In \cite{FEJER}, Fej\'{e}r gave a generalization of the inequalities (\ref%
{HH}) as the following:

If $f:\left[ a,b\right] \rightarrow 
\mathbb{R}
$ is a convex function, and $g:\left[ a,b\right] \rightarrow 
\mathbb{R}
$ is nonnegative, integrable and symmetric about $\frac{a+b}{2},$ then 
\begin{equation}
f\left( \frac{a+b}{2}\right) \int_{a}^{b}g\left( x\right) dx\leq
\int_{a}^{b}f\left( x\right) g\left( x\right) dx\leq \frac{f\left( a\right)
+f\left( b\right) }{2}\int_{a}^{b}g\left( x\right) dx\text{ }.  \label{FEJ}
\end{equation}

If $g\equiv 1$, then we are talking about the Hermite--Hadamard
inequalities.More about those inequalities can be found in a number of
papers and monographies (for example, see \cite{DRAG}-\cite{e.set}).

In \cite{HUDMA}, Hudzik and Maligrada considered among others the class of
functions which are s-convex in the second sense.

\begin{definition}
\label{d1} A function $f:[0,\infty )\rightarrow 
\mathbb{R}
$ is said to be s-convex in the second sense if%
\begin{equation}
f\left( \lambda x+\left( 1-\lambda \right) y\right) \leq \lambda ^{s}f\left(
x\right) +\left( 1-\lambda \right) ^{s}f\left( y\right) .  \label{scon}
\end{equation}%
for all $x,y\in \lbrack 0,\infty ),$ $\lambda \in \left[ 0,1\right] $ and
for some fixed $s\in (0,1].$
\end{definition}

It can be easily seen that $s=1,$ s-convexity reduces to ordinary convexity
of functions defined on $[0,\infty ).$

In \cite{DRAG}, Dragomir and Fitzpatrick proved Hadamard's inequality \
which holds \ for s-convex functions in the second sense.

\begin{theorem}
\label{t1} Suppose that $f:[0,\infty )\rightarrow \lbrack 0,\infty )$ is an
s-convex functions in the second sense, where $s\in \left( 0,1\right) ,$ and
let $a,b\in \lbrack 0,\infty ),$ $a<b.$ If \ $f\in L\left[ a,b\right] ,$
then the following inequalities hold:%
\begin{equation}
2^{s-1}f\left( \frac{a+b}{2}\right) \leq \frac{1}{b-a}\int_{a}^{b}f\left(
x\right) dx\leq \frac{f\left( a\right) +f\left( b\right) }{s+1}  \label{shad}
\end{equation}
\end{theorem}

We give some necessary definitions and \ mathematical preliminaries of
fractional calculus theory which are used throught this paper.

\begin{definition}
\label{d2} Let $f\in \left[ a,b\right] .$ The $Riemann-Liouville$ integrals $%
J_{a+}^{\alpha }f\ and\ J_{b-}^{\alpha }f$ of order $\alpha >0$ with $a\geq
0 $ are defined by integrals hold:%
\begin{equation}
J_{a+}^{\alpha }f(x)=\frac{1}{\Gamma \left( \alpha \right) }%
\int_{a}^{x}\left( x-t\right) ^{\alpha -1}f\left( t\right) dt,\ \ x>a
\label{J integral}
\end{equation}%
and%
\begin{equation*}
J_{b-}^{\alpha }f(x)=\frac{1}{\Gamma \left( \alpha \right) }%
\int_{x}^{b}\left( t-x\right) ^{\alpha -1}f\left( t\right) dt,\ \ x<b
\end{equation*}%
respectively where $\Gamma \left( \alpha \right) =\int_{0}^{\infty
}e^{-t}u^{\alpha -1}du$. Here is $J_{a+}^{0}f(x)=J_{b-}^{0}f(x)=f\left(
x\right) $ \newline
\end{definition}

In the case of $\alpha =1,$ the fractional integral reduces to the classical
integral. The recent results and the properties concerning this operator can
be found (\cite{Anastassiou}-\cite{Dahmani4})

In \cite{sarikaya}, Sar\i kaya \textit{et.al.} represented
Hermite-Hadamard's inequalities in fractional integral forms as follows.

\begin{theorem}
\label{t2} Let $f:\left[ a,b\right] \rightarrow 
\mathbb{R}
$ be positive function with $0\leq a<b$ and $f\in L\left[ a,b\right] .$ If $%
f $ is a convex function on $\left[ a,b\right] ,$ than the following
inequalities for fractional integrals hold%
\begin{equation}
f\left( \frac{a+b}{2}\right) \leq \frac{\Gamma \left( \alpha +1\right) }{%
2\left( b-a\right) ^{\alpha }}\left[ J_{a+}^{\alpha }f\left( b\right)
+J_{b-}^{\alpha }f\left( a\right) \right] \leq \frac{f\left( a\right)
+f\left( b\right) }{2}  \label{sarikaya}
\end{equation}%
\newline
with $\alpha >0.$
\end{theorem}

In \cite{i.iscan}, \.{I}\c{s}can gave \ the following Hermite-Hadamard-Fejer
integral inequalities via fractional integrals:

\begin{theorem}
\label{t3} Let $f:\left[ a,b\right] \rightarrow 
\mathbb{R}
$ be convex function with $a<b$ and $f\in L\left[ a,b\right] .$If $g:\left[
a,b\right] \rightarrow 
\mathbb{R}
$ is nonnegative, integrable and symmetric to $\left( a+b\right) /2,$ then \
the following inequalities for fractional integrals hold%
\begin{eqnarray*}
f\left( \frac{a+b}{2}\right) \left[ J_{a+}^{\alpha }g\left( b\right)
+J_{b-}^{\alpha }g\left( a\right) \right] &\leq &\left[ J_{a+}^{\alpha
}\left( fg\right) \left( b\right) +J_{b-}^{\alpha }\left( fg\right) \left(
a\right) \right] \\
&\leq &\frac{f\left( a\right) +f\left( b\right) }{2}\left[ J_{a+}^{\alpha
}g\left( b\right) +J_{b-}^{\alpha }g\left( a\right) \right]
\end{eqnarray*}%
with $\alpha >0.$
\end{theorem}

Set \textit{et al}. established some inequalities connected with the
left-hand side of the inequality (\ref{FEJ}) used the following lemma.

\begin{lemma}
\label{L1} \cite{e.set} Let $f:\left[ a,b\right] \rightarrow 
\mathbb{R}
$ be a differentiable mapping on $\left( a,b\right) $ with $a<b$ and let $g:%
\left[ a,b\right] \rightarrow 
\mathbb{R}
$. If $f^{\prime },g\in L\left[ a,b\right] $, then the following identity
for fractional integrals holds:%
\begin{eqnarray}
&&f\left( \frac{a+b}{2}\right) \left[ J_{\left( \frac{a+b}{2}\right)
-}^{\alpha }g\left( a\right) +J_{\left( \frac{a+b}{2}\right) +}^{\alpha
}g\left( b\right) \right]  \label{e. lemma} \\
&&-\left[ J_{\left( \frac{a+b}{2}\right) -}^{\alpha }\left( fg\right) \left(
a\right) +J_{\left( \frac{a+b}{2}\right) +}^{\alpha }\left( fg\right) \left(
b\right) \right]  \notag \\
&=&\frac{1}{\Gamma \left( \alpha \right) }\int_{a}^{b}k\left( t\right)
f^{\prime }\left( t\right) dt  \notag
\end{eqnarray}%
where%
\begin{equation*}
k\left( t\right) =\left\{ 
\begin{array}{c}
\int\limits_{a}^{t}\left( s-a\right) ^{\alpha -1}g\left( s\right) ds,\text{
\ \ \ \ \ \ }t\in \left[ a,\frac{a+b}{2}\right] \\ 
-\int\limits_{t}^{b}\left( b-s\right) ^{\alpha -1}g\left( s\right) ds,\text{
\ \ \ \ }t\in \left[ \frac{a+b}{2},b\right]%
\end{array}%
\right. .
\end{equation*}
\end{lemma}

Set \textit{et al. proved the following three theorems.}

\begin{theorem}
\label{t4} \cite{e.set}Let $f:I\rightarrow 
\mathbb{R}
$ be a differentiable mapping on $I^{\circ }$and $f^{\prime }\in L\left[ a,b%
\right] $ with $a<b$ and $g:\left[ a,b\right] \rightarrow 
\mathbb{R}
$ is continuous. If $\left\vert f^{\prime }\right\vert $ is convex on $\left[
a,b\right] ,$ then the following inequality for fractional integrals holds:%
\begin{eqnarray}
&&\left\vert f\left( \frac{a+b}{2}\right) \left[ J_{\left( \frac{a+b}{2}%
\right) -}^{\alpha }g\left( a\right) +J_{\left( \frac{a+b}{2}\right)
+}^{\alpha }g\left( b\right) \right] \right.  \label{6,0} \\
&&\left. -\left[ J_{\left( \frac{a+b}{2}\right) -}^{\alpha }\left( fg\right)
\left( a\right) +J_{\left( \frac{a+b}{2}\right) +}^{\alpha }\left( fg\right)
\left( b\right) \right] \right\vert  \notag \\
&\leq &\frac{\left( b-a\right) ^{\alpha +1}\left\vert \left\vert
g\right\vert \right\vert _{\left[ a,b\right] ,\infty }}{2^{\alpha +1}\Gamma
\left( \alpha +1\right) \left( \alpha +1\right) }\left[ \left\vert f^{\prime
}\left( a\right) \right\vert +\left\vert f^{\prime }\left( b\right)
\right\vert \right]  \notag
\end{eqnarray}%
with $\alpha >0.$
\end{theorem}

\begin{theorem}
\label{t5} \cite{e.set} Let $f:I\rightarrow 
\mathbb{R}
$ be a differentiable mapping on $I^{\circ }$and $f^{\prime }\in L\left[ a,b%
\right] $ with $a<b$ and $g:\left[ a,b\right] \rightarrow 
\mathbb{R}
$ is continuous. If $\left\vert f^{\prime }\right\vert ^{q}$ is convex on $%
\left[ a,b\right] ,$ $q>1$, then the following inequality for fractional
integrals holds:%
\begin{eqnarray}
&&  \label{7.0} \\
&&\left\vert f\left( \frac{a+b}{2}\right) \left[ J_{\left( \frac{a+b}{2}%
\right) -}^{\alpha }g\left( a\right) +J_{\left( \frac{a+b}{2}\right)
+}^{\alpha }g\left( b\right) \right] \right.  \notag \\
&&\left. -\left[ J_{\left( \frac{a+b}{2}\right) -}^{\alpha }\left( fg\right)
\left( a\right) +J_{\left( \frac{a+b}{2}\right) +}^{\alpha }\left( fg\right)
\left( b\right) \right] \right\vert  \notag \\
&\leq &\frac{\left( b-a\right) ^{\alpha +1}\left\vert \left\vert
g\right\vert \right\vert _{\left[ a,b\right] ,\infty }}{2^{\alpha
+1+1/q}\left( \alpha +1\right) \left( \alpha +2\right) ^{1/q}\Gamma \left(
\alpha +1\right) }\left\{ \left( \left( \alpha +3\right) \left\vert
f^{\prime }\left( a\right) \right\vert ^{q}+\left( \alpha +1\right)
\left\vert f^{\prime }\left( b\right) \right\vert ^{q}\right) ^{1/q}\right. 
\notag \\
&&\left. +\left( \left( \alpha +1\right) \left\vert f^{\prime }\left(
a\right) \right\vert ^{q}+\left( \alpha +3\right) \left\vert f^{\prime
}\left( b\right) \right\vert ^{q}\right) ^{1/q}\right\}  \notag
\end{eqnarray}%
with $\alpha >0.$
\end{theorem}

\begin{theorem}
\label{theorem 8.0} \cite{e.set} Let $f:I\rightarrow 
\mathbb{R}
$ be a differentiable mapping on $I^{\circ }$and $f^{\prime }\in L\left[ a,b%
\right] $ with $a<b$ and $g:\left[ a,b\right] \rightarrow 
\mathbb{R}
$ is continuous. If $\left\vert f^{\prime }\right\vert ^{q}$ is convex on $%
\left[ a,b\right] ,$ $q>1$, then the following inequality for fractional
integrals holds:%
\begin{eqnarray}
&&\left\vert f\left( \frac{a+b}{2}\right) \left[ J_{\left( \frac{a+b}{2}%
\right) -}^{\alpha }g\left( a\right) +J_{\left( \frac{a+b}{2}\right)
+}^{\alpha }g\left( b\right) \right] \right.  \label{8,0} \\
&&\left. -\left[ J_{\left( \frac{a+b}{2}\right) -}^{\alpha }\left( fg\right)
\left( a\right) +J_{\left( \frac{a+b}{2}\right) +}^{\alpha }\left( fg\right)
\left( b\right) \right] \right\vert  \notag \\
&\leq &\frac{\left( b-a\right) ^{\alpha +1}\left\vert \left\vert
g\right\vert \right\vert _{\infty }}{2^{\alpha +1+2/q}\left( \alpha
p+1\right) ^{1/q}\Gamma \left( \alpha +1\right) }  \notag \\
&&\times \left[ 3\left\vert f^{\prime }\left( a\right) \right\vert
^{q}+\left\vert f^{\prime }\left( b\right) \right\vert ^{q1/q}+\left(
\left\vert f^{\prime }\left( a\right) \right\vert ^{q}+3\left\vert f^{\prime
}\left( b\right) \right\vert ^{q}\right) ^{1/q}\right]  \notag
\end{eqnarray}%
\newline
where $1/p+1/q=1.$\newline
\end{theorem}

We recall the following function:

The incomplete Beta function by 
\begin{equation*}
B_{x}\left( \alpha ,\beta \right) =\int_{0}^{x}t^{\alpha -1}\left(
1-t\right) ^{\beta -1}dt.
\end{equation*}

In this paper, motivated by the recent results given in \cite{HUDMA},\cite%
{e.set}, we establish Hermite-Hadamard-Fejer type inequalities for s-convex
functions in the second sense via fractional integral. An interesting
feature of our results is that they provide new estimates on these types of
inequalities for fractional integrals.

\section{MA\.{I}N RESULTS}

Now, by using the Lemma $\ref{L1}$ we prove our main theorems.

\begin{theorem}
\label{T6}Let $f:I\subseteq \lbrack 0,\infty )\rightarrow 
\mathbb{R}
$ be a differentiable mapping on $I^{o}$ and $f^{\prime }\in L\left[ a,b%
\right] $ with $a<b$ and $g:\left[ a,b\right] \rightarrow 
\mathbb{R}
$ is continious. If $\left\vert f^{\prime }\right\vert $ is $s-convex$ on $%
\left[ a,b\right] $ for some fixed $s\in (0,1],$ then the following
inequality for fractional integrals holds:%
\begin{eqnarray}
&&\left\vert f\left( \frac{a+b}{2}\right) \left[ J_{\left( \frac{a+b}{2}%
\right) -}^{\alpha }g\left( a\right) +J_{\left( \frac{a+b}{2}\right)
+}^{\alpha }g\left( b\right) \right] \right.  \label{t61} \\
&&\left. -\left[ J_{\left( \frac{a+b}{2}\right) -}^{\alpha }\left( fg\right)
\left( a\right) +J_{\left( \frac{a+b}{2}\right) +}^{\alpha }\left( fg\right)
\left( b\right) \right] \right\vert  \notag \\
&\leq &\frac{\left( b-a\right) ^{\alpha +1}\left\vert \left\vert
g\right\vert \right\vert _{\left[ a,b\right] ,\infty }}{\Gamma \left( \alpha
+1\right) }  \notag \\
&&\times \left\{ B_{1/2}\left( \alpha +1,s+1\right) +\frac{1}{2^{\alpha
+s+1}\left( \alpha +s+1\right) }\right\} \left[ \left\vert f^{\prime }\left(
a\right) \right\vert +\left\vert f^{\prime }\left( b\right) \right\vert %
\right] .  \notag
\end{eqnarray}
\end{theorem}

\begin{proof}
Since $\left\vert f^{\prime }\right\vert $ is $s-convex$ on $\left[ a,b%
\right] $ for some fixed $s\in (0,1]$, we know that for $t\in \left[ a,b%
\right] $%
\begin{equation*}
\left\vert f^{\prime }\left( t\right) \right\vert =\left\vert f^{\prime
}\left( \frac{b-t}{b-a}a+\frac{t-a}{b-a}b\right) \right\vert \leq \left( 
\frac{b-t}{b-a}\right) ^{s}\left\vert f^{\prime }\left( a\right) \right\vert
+\left( \frac{t-a}{b-a}\right) ^{s}\left\vert f^{\prime }\left( b\right)
\right\vert
\end{equation*}%
From Lemma \ref{L1} we have%
\begin{eqnarray*}
&&\left\vert f\left( \frac{a+b}{2}\right) \left[ J_{\left( \frac{a+b}{2}%
\right) -}^{\alpha }g\left( a\right) +J_{\left( \frac{a+b}{2}\right)
+}^{\alpha }g\left( b\right) \right] \right. \\
&&\left. -\left[ J_{\left( \frac{a+b}{2}\right) -}^{\alpha }\left( fg\right)
\left( a\right) +J_{\left( \frac{a+b}{2}\right) +}^{\alpha }\left( fg\right)
\left( b\right) \right] \right\vert \\
&\leq &\frac{1}{\Gamma \left( \alpha \right) }\left\{ \dint_{a}^{\frac{a+b}{2%
}}\left\vert \dint_{a}^{t}\left( s-a\right) ^{\alpha -1}g\left( s\right)
ds\right\vert \left\vert f^{\prime }\left( t\right) \right\vert dt\right. \\
&&\left. +\dint_{\frac{a+b}{2}}^{b}\left\vert \dint_{t}^{b}\left( b-s\right)
^{\alpha -1}g\left( s\right) ds\right\vert \left\vert f^{\prime }\left(
t\right) \right\vert dt\right\}
\end{eqnarray*}%
\begin{eqnarray*}
&\leq &\frac{\left\vert \left\vert g\right\vert \right\vert _{\left[ a,\frac{%
a+b}{2}\right] ,\infty }}{\left( b-a\right) ^{s}\Gamma \left( \alpha \right) 
}\dint_{a}^{\frac{a+b}{2}}\left( \dint_{a}^{t}\left( s-a\right) ^{\alpha
-1}ds\right) \left( \left( b-t\right) ^{s}\left\vert f^{\prime }\left(
a\right) \right\vert +\left( t-a\right) ^{s}\left\vert f^{\prime }\left(
b\right) \right\vert \right) dt \\
&&+\frac{\left\vert \left\vert g\right\vert \right\vert _{\left[ \frac{a+b}{2%
},b\right] ,\infty }}{\left( b-a\right) ^{s}\Gamma \left( \alpha \right) }%
\dint_{\frac{a+b}{2}}^{b}\left( \dint_{t}^{b}\left( b-s\right) ^{\alpha
-1}ds\right) \left( \left( b-t\right) ^{s}\left\vert f^{\prime }\left(
a\right) \right\vert +\left( t-a\right) ^{s}\left\vert f^{\prime }\left(
b\right) \right\vert \right) dt \\
&=&\frac{\left\vert \left\vert g\right\vert \right\vert _{\left[ a,\frac{a+b%
}{2}\right] ,\infty }}{\left( b-a\right) ^{s}\Gamma \left( \alpha +1\right) }%
\dint_{a}^{\frac{a+b}{2}}\left( t-a\right) ^{\alpha }\left[ \left(
b-t\right) ^{s}\left\vert f^{\prime }\left( a\right) \right\vert +\left(
t-a\right) ^{s}\left\vert f^{\prime }\left( b\right) \right\vert \right] dt
\end{eqnarray*}%
\begin{eqnarray*}
&&+\frac{\left\vert \left\vert g\right\vert \right\vert _{\left[ \frac{a+b}{2%
},b\right] ,\infty }}{\left( b-a\right) ^{s}\Gamma \left( \alpha +1\right) }%
\dint_{\frac{a+b}{2}}^{b}\left( b-t\right) ^{\alpha }\left[ \left(
b-t\right) ^{s}\left\vert f^{\prime }\left( a\right) \right\vert +\left(
t-a\right) ^{s}\left\vert f^{\prime }\left( b\right) \right\vert \right] dt
\\
&=&\frac{\left\vert \left\vert g\right\vert \right\vert _{\left[ a,\frac{a+b%
}{2}\right] ,\infty }}{\left( b-a\right) ^{s}\Gamma \left( \alpha +1\right) }%
\left[ \left\vert f^{\prime }\left( a\right) \right\vert \dint_{a}^{\frac{a+b%
}{2}}\left( t-a\right) ^{\alpha }\left( b-t\right) ^{s}dt+\left\vert
f^{\prime }\left( b\right) \right\vert \dint_{a}^{\frac{a+b}{2}}\left(
t-a\right) ^{\alpha +s}dt\right] \\
&&+\frac{\left\vert \left\vert g\right\vert \right\vert _{\left[ \frac{a+b}{2%
},b\right] ,\infty }}{\left( b-a\right) ^{s}\Gamma \left( \alpha +1\right) }%
\left[ \left\vert f^{\prime }\left( a\right) \right\vert \dint_{\frac{a+b}{2}%
}^{b}\left( b-t\right) ^{\alpha +s}dt+\left\vert f^{\prime }\left( b\right)
\right\vert \dint_{\frac{a+b}{2}}^{b}\left( b-t\right) ^{\alpha }\left(
t-a\right) ^{s}dt\right]
\end{eqnarray*}%
\begin{eqnarray*}
&=&\frac{\left\vert \left\vert g\right\vert \right\vert _{\left[ a,\frac{a+b%
}{2}\right] ,\infty }}{\left( b-a\right) ^{s}\Gamma \left( \alpha +1\right) }
\\
&&\times \left[ \left\vert f^{\prime }\left( a\right) \right\vert \left(
b-a\right) ^{\alpha +s+1}B_{1/2}\left( \alpha +1,s+1\right) +\left\vert
f^{\prime }\left( b\right) \right\vert \frac{\left( b-a\right) ^{\alpha +s+1}%
}{2^{\alpha +s+1}\left( \alpha +s+1\right) }\right] \\
&&+\frac{\left\vert \left\vert g\right\vert \right\vert _{\left[ \frac{a+b}{2%
},b\right] ,\infty }}{\left( b-a\right) ^{s}\Gamma \left( \alpha +1\right) }
\\
&&\times \left[ \left\vert f^{\prime }\left( a\right) \right\vert \frac{%
\left( b-a\right) ^{\alpha +s+1}}{2^{\alpha +s+1}\left( \alpha +s+1\right) }%
+\left\vert f^{\prime }\left( b\right) \right\vert \left( b-a\right)
^{\alpha +s+1}B_{1/2}\left( \alpha +1,s+1\right) \right] \\
&=&\frac{\left( b-a\right) ^{\alpha +s+1}}{\left( b-a\right) ^{s}\Gamma
\left( \alpha +1\right) } \\
&&\times \left\{ \left\vert \left\vert g\right\vert \right\vert _{\left[ a,%
\frac{a+b}{2}\right] ,\infty }\left( \left\vert f^{\prime }\left( a\right)
\right\vert B_{1/2}\left( \alpha +1,s+1\right) +\left\vert f^{\prime }\left(
b\right) \right\vert \frac{1}{2^{\alpha +s+1}\left( \alpha +s+1\right) }%
\right) \right. \\
&&\left. +\left\vert \left\vert g\right\vert \right\vert _{\left[ \frac{a+b}{%
2},b\right] ,\infty }\left( \left\vert f^{\prime }\left( a\right)
\right\vert \frac{1}{2^{\alpha +s+1}\left( \alpha +s+1\right) }+\left\vert
f^{\prime }\left( b\right) \right\vert B_{1/2}\left( \alpha +1,s+1\right)
\right) \right\}
\end{eqnarray*}%
\begin{eqnarray*}
&\leq &\frac{\left( b-a\right) ^{\alpha +1}\left\vert \left\vert
g\right\vert \right\vert _{\left[ a,b\right] ,\infty }}{\Gamma \left( \alpha
+1\right) } \\
&&\times \left\{ B_{1/2}\left( \alpha +1,s+1\right) \left[ \left\vert
f^{\prime }\left( a\right) \right\vert +\left\vert f^{\prime }\left(
b\right) \right\vert \right] +\frac{1}{2^{\alpha +s+1}\left( \alpha
+s+1\right) }\left[ \left\vert f^{\prime }\left( a\right) \right\vert
+\left\vert f^{\prime }\left( b\right) \right\vert \right] \right\} \\
&=&\frac{\left( b-a\right) ^{\alpha +1}\left\vert \left\vert g\right\vert
\right\vert _{\left[ a,b\right] ,\infty }}{\Gamma \left( \alpha +1\right) }
\\
&&\times \left\{ B_{1/2}\left( \alpha +1,s+1\right) +\frac{1}{2^{\alpha
+s+1}\left( \alpha +s+1\right) }\right\} \left[ \left\vert f^{\prime }\left(
a\right) \right\vert +\left\vert f^{\prime }\left( b\right) \right\vert %
\right]
\end{eqnarray*}%
where%
\begin{equation*}
\dint_{a}^{\frac{a+b}{2}}\left( t-a\right) ^{\alpha +s}dt=\dint_{\frac{a+b}{2%
}}^{b}\left( b-t\right) ^{\alpha +s}dt=\frac{\left( b-a\right) ^{\alpha +s+1}%
}{2^{\alpha +s+1}\left( \alpha +s+1\right) }
\end{equation*}%
and%
\begin{eqnarray*}
\dint_{a}^{\frac{a+b}{2}}\left( t-a\right) ^{\alpha }\left( b-t\right)
^{s}dt &=&\dint_{\frac{a+b}{2}}^{b}\left( b-t\right) ^{\alpha }\left(
t-a\right) ^{s}dt \\
&=&\left( b-a\right) ^{\alpha +s+1}B_{1/2}\left( \alpha +1,s+1\right) .
\end{eqnarray*}
\end{proof}

\begin{remark}
In Theorem 7, if we choose $s=1$ , then $\left( \ref{t61}\right) $ reduces
inequality ($\ref{6,0}$) of Theorem 4.
\end{remark}

\begin{theorem}
\label{T7} Let $f:I\subseteq \lbrack 0,\infty )\rightarrow 
\mathbb{R}
$ be a differentiable mapping on $I^{o}$ and $f^{\prime }\in L\left[ a,b%
\right] $ with $a<b$ and let $g:\left[ a,b\right] \rightarrow 
\mathbb{R}
$ is continious. If $\left\vert f^{\prime }\right\vert ^{q}$ is $s-convex$
on $\left[ a,b\right] $ for some fixed $s\in (0,1],$ $q>1$, then the
following inequality for fractional integrals holds:
\end{theorem}

\begin{eqnarray*}
&&\left\vert f\left( \frac{a+b}{2}\right) \left[ J_{\left( \frac{a+b}{2}%
\right) -}^{\alpha }g\left( a\right) +J_{\left( \frac{a+b}{2}\right)
+}^{\alpha }g\left( b\right) \right] \right. \\
&&\left. -\left[ J_{\left( \frac{a+b}{2}\right) -}^{\alpha }\left( fg\right)
\left( a\right) +J_{\left( \frac{a+b}{2}\right) +}^{\alpha }\left( fg\right)
\left( b\right) \right] \right\vert \\
&\leq &\frac{\left( b-a\right) ^{\alpha +1}\left\vert \left\vert
g\right\vert \right\vert _{\left[ a,b\right] ,\infty }}{2^{\alpha +1+\frac{1%
}{q}}\left( \alpha +1\right) \left( \alpha +2\right) ^{1/q}\left( \alpha
+s+q\right) ^{1/q}\Gamma \left( \alpha +1\right) }
\end{eqnarray*}%
\begin{eqnarray}
&&\times \left\{ \left( \left( \alpha +s+1\right) \left( \alpha +3\right)
\left\vert f^{\prime }\left( a\right) \right\vert ^{q}+2^{1-s}\left( \alpha
+1\right) \left( \alpha +2\right) \left\vert f^{\prime }\left( b\right)
\right\vert ^{q}\right) ^{1/q}\right.  \label{T71} \\
&&+\left. \left( 2^{1-s}\left( \alpha +1\right) \left( \alpha +2\right)
\left\vert f^{\prime }\left( a\right) \right\vert ^{q}+\left( \alpha
+s+1\right) \left( \alpha +3\right) \left\vert f^{\prime }\left( b\right)
\right\vert ^{q}\right) ^{1/q}\right\} .  \notag
\end{eqnarray}

\begin{proof}
Since $\left\vert f^{\prime }\right\vert $ is $s-convex$ on $\left[ a,b%
\right] $ for some fixed $s\in (0,1]$, we know that for $t\in \left[ a,b%
\right] $%
\begin{equation*}
\left\vert f^{\prime }\left( t\right) \right\vert ^{q}=\left\vert f^{\prime
}\left( \frac{b-t}{b-a}a+\frac{t-a}{b-a}b\right) \right\vert ^{q}\leq \left( 
\frac{b-t}{b-a}\right) ^{s}\left\vert f^{\prime }\left( a\right) \right\vert
^{q}+\left( \frac{t-a}{b-a}\right) ^{s}\left\vert f^{\prime }\left( b\right)
\right\vert ^{q}
\end{equation*}%
Using Lemma \ref{L1}, power mean inequality and convexity of $\left\vert
f^{\prime }\right\vert ^{q},$ it follows that%
\begin{eqnarray*}
&&\left\vert f\left( \frac{a+b}{2}\right) \left[ J_{\left( \frac{a+b}{2}%
\right) -}^{\alpha }g\left( a\right) +J_{\left( \frac{a+b}{2}\right)
+}^{\alpha }g\left( b\right) \right] \right. \\
&&\left. -\left[ J_{\left( \frac{a+b}{2}\right) -}^{\alpha }\left( fg\right)
\left( a\right) +J_{\left( \frac{a+b}{2}\right) +}^{\alpha }\left( fg\right)
\left( b\right) \right] \right\vert \\
&\leq &\frac{1}{\Gamma \left( \alpha \right) }\left( \dint_{a}^{\frac{a+b}{2}%
}\left\vert \dint_{a}^{t}\left( s-a\right) ^{\alpha -1}g\left( s\right)
ds\right\vert dt\right) ^{1-1/q} \\
&&\times \left( \dint_{a}^{\frac{a+b}{2}}\left\vert \dint_{a}^{t}\left(
s-a\right) ^{\alpha -1}g\left( s\right) ds\right\vert \left\vert f^{\prime
}\left( t\right) \right\vert ^{q}dt\right) ^{1/q} \\
&&+\frac{1}{\Gamma \left( \alpha \right) }\left( \dint_{\frac{a+b}{2}%
}^{b}\left\vert \dint_{t}^{b}\left( b-s\right) ^{\alpha -1}g\left( s\right)
ds\right\vert dt\right) ^{1-1/q} \\
&&\times \left( \dint_{\frac{a+b}{2}}^{b}\left\vert \dint_{t}^{b}\left(
b-s\right) ^{\alpha -1}g\left( s\right) ds\right\vert \left\vert f^{\prime
}\left( t\right) \right\vert ^{q}dt\right) ^{1/q} \\
&\leq &\frac{\left\vert \left\vert g\right\vert \right\vert _{\left[ a,\frac{%
a+b}{2}\right] ,\infty }}{\Gamma \left( \alpha \right) }\left( \dint_{a}^{%
\frac{a+b}{2}}\left\vert \dint_{a}^{t}\left( s-a\right) ^{\alpha
-1}ds\right\vert dt\right) ^{1-1/q} \\
&&\times \left( \dint_{a}^{\frac{a+b}{2}}\left\vert \dint_{a}^{t}\left(
s-a\right) ^{\alpha -1}ds\right\vert \left\vert f^{\prime }\left( t\right)
\right\vert ^{q}dt\right) ^{1/q}
\end{eqnarray*}%
\begin{eqnarray*}
&&+\frac{\left\vert \left\vert g\right\vert \right\vert _{\left[ \frac{a+b}{2%
},b\right] ,\infty }}{\Gamma \left( \alpha \right) }\left( \dint_{\frac{a+b}{%
2}}^{b}\left\vert \dint_{t}^{b}\left( b-s\right) ^{\alpha -1}ds\right\vert
dt\right) ^{1-1/q} \\
&&\times \left( \dint_{\frac{a+b}{2}}^{b}\left\vert \dint_{t}^{b}\left(
b-s\right) ^{\alpha -1}ds\right\vert \left\vert f^{\prime }\left( t\right)
\right\vert ^{q}dt\right) ^{1/q}
\end{eqnarray*}%
\begin{eqnarray*}
&\leq &\frac{1}{\alpha \Gamma \left( \alpha \right) }\left( \frac{\left(
b-a\right) ^{\alpha +1}}{2^{\alpha +1}\left( \alpha +1\right) }\right)
^{1-1/q} \\
&&\times \left\{ \frac{\left\vert \left\vert g\right\vert \right\vert _{%
\left[ a,\frac{a+b}{2}\right] ,\infty }}{\left( b-a\right) ^{s/q}}\left[
\dint_{a}^{\frac{a+b}{2}}\left( \left( t-a\right) ^{\alpha }\left(
b-t\right) ^{s}\left\vert f^{\prime }\left( a\right) \right\vert ^{q}+\left(
t-a\right) ^{\alpha +s}\left\vert f^{\prime }\left( b\right) \right\vert
^{q}\right) dt\right] ^{1/q}\right.
\end{eqnarray*}%
\begin{eqnarray*}
&&+\left. \frac{\left\vert \left\vert g\right\vert \right\vert _{\left[ 
\frac{a+b}{2},b\right] ,\infty }}{\left( b-a\right) ^{s/q}}\left[ \dint_{%
\frac{a+b}{2}}^{b}\left( \left( b-t\right) ^{\alpha +s}\left\vert f^{\prime
}\left( a\right) \right\vert ^{q}+\left( b-t\right) ^{\alpha }\left(
t-a\right) ^{s}\left\vert f^{\prime }\left( b\right) \right\vert ^{q}\right)
dt\right] ^{1/q}\right\} \\
&=&\frac{1}{\Gamma \left( \alpha +1\right) }\left( \frac{\left( b-a\right)
^{\alpha +1}}{2^{\alpha +1}\left( \alpha +1\right) }\right) ^{1-1/q}\left\{
\left( \frac{\left\vert \left\vert g\right\vert \right\vert _{\left[ a,\frac{%
a+b}{2}\right] ,\infty }}{\left( b-a\right) ^{s/q}}\right. \right.
\end{eqnarray*}%
\begin{eqnarray*}
&&\left. \times \left[ \left( \left( b-a\right) ^{\alpha +s+1}B_{1/2}\left(
\alpha +1,s+1\right) \left\vert f^{\prime }\left( a\right) \right\vert ^{q}+%
\frac{\left( b-a\right) ^{\alpha +s+1}}{2^{\alpha +s+1}\left( \alpha
+s+1\right) }\left\vert f^{\prime }\left( b\right) \right\vert ^{q}\right) dt%
\right] ^{1/q}\right) \\
&&+\left( \frac{\left\vert \left\vert g\right\vert \right\vert _{\left[ 
\frac{a+b}{2},b\right] ,\infty }}{\left( b-a\right) ^{s/q}}\right. \\
&&\left. \left. \left[ \times \left( \frac{\left( b-a\right) ^{\alpha +s+1}}{%
2^{\alpha +s+1}\left( \alpha +s+1\right) }\left\vert f^{\prime }\left(
a\right) \right\vert ^{q}+\left( b-a\right) ^{\alpha +s+1}B_{1/2}\left(
\alpha +1,s+1\right) \left\vert f^{\prime }\left( b\right) \right\vert
^{q}\right) dt\right] ^{1/q}\right) \right\}
\end{eqnarray*}%
\begin{eqnarray*}
&\leq &\frac{\left( b-a\right) ^{\alpha +1}\left\vert \left\vert
g\right\vert \right\vert _{\left[ a,b\right] ,\infty }}{2^{\alpha +1+\frac{s%
}{q}}\left( \alpha +1\right) ^{1-1/q}\left( \alpha +s+q\right) ^{1/q}\Gamma
\left( \alpha +1\right) } \\
&&\times \left\{ \left( 2^{\alpha +s+1}\left( \alpha +s+1\right)
B_{1/2}\left( \alpha +1,s+1\right) \left\vert f^{\prime }\left( a\right)
\right\vert ^{q}+\left\vert f^{\prime }\left( b\right) \right\vert
^{q}\right) ^{1/q}\right.
\end{eqnarray*}%
\begin{eqnarray*}
&&\left. \left( \left\vert f^{\prime }\left( a\right) \right\vert
^{q}+2^{\alpha +s+1}\left( \alpha +s+1\right) B_{1/2}\left( \alpha
+1,s+1\right) \left\vert f^{\prime }\left( b\right) \right\vert ^{q}\right)
^{1/q}\right\} \\
&\leq &\frac{\left( b-a\right) ^{\alpha +1}\left\vert \left\vert
g\right\vert \right\vert _{\left[ a,b\right] ,\infty }}{2^{\alpha +1+\frac{1%
}{q}}\left( \alpha +1\right) \left( \alpha +2\right) ^{1/q}\left( \alpha
+s+q\right) ^{1/q}\Gamma \left( \alpha +1\right) }
\end{eqnarray*}%
\begin{eqnarray*}
&&\times \left\{ \left( \left( \alpha +s+1\right) \left( \alpha +3\right)
\left\vert f^{\prime }\left( a\right) \right\vert ^{q}+2^{1-s}\left( \alpha
+1\right) \left( \alpha +2\right) \left\vert f^{\prime }\left( b\right)
\right\vert ^{q}\right) ^{1/q}\right. \\
&&+\left. \left( 2^{1-s}\left( \alpha +1\right) \left( \alpha +2\right)
\left\vert f^{\prime }\left( a\right) \right\vert ^{q}+\left( \alpha
+s+1\right) \left( \alpha +3\right) \left\vert f^{\prime }\left( b\right)
\right\vert ^{q}\right) ^{1/q}\right\} .
\end{eqnarray*}
\end{proof}

\begin{remark}
In Theorem 8, if we choose $s=1$ , then $\left( \ref{T71}\right) $ reduces
inequality $\left( \ref{7.0}\right) \;$of Theorem 5.
\end{remark}

\begin{theorem}
\label{T8}Let $f:I\subseteq \lbrack 0,\infty )\rightarrow 
\mathbb{R}
$ be a differentiable mapping on $I^{o}$ and $f^{\prime }\in L\left[ a,b%
\right] $ with $a<b$ and let $g:\left[ a,b\right] \rightarrow 
\mathbb{R}
$ is continious. If $\left\vert f^{\prime }\right\vert ^{q}$ is $s-convex$
on $\left[ a,b\right] $ \ for some fixed $s\in (0,1],$ $q>1$, then the
following inequality for fractional integrals holds:%
\begin{eqnarray}
&&\left\vert f\left( \frac{a+b}{2}\right) \left[ J_{\left( \frac{a+b}{2}%
\right) -}^{\alpha }g\left( a\right) +J_{\left( \frac{a+b}{2}\right)
+}^{\alpha }g\left( b\right) \right] \right.  \label{T81} \\
&&\left. -\left[ J_{\left( \frac{a+b}{2}\right) -}^{\alpha }\left( fg\right)
\left( a\right) +J_{\left( \frac{a+b}{2}\right) +}^{\alpha }\left( fg\right)
\left( b\right) \right] \right\vert  \notag \\
&\leq &\frac{\left( b-a\right) ^{\alpha +1}\left\vert \left\vert
g\right\vert \right\vert _{\left[ a,b\right] ,\infty }}{2^{\alpha +1+\frac{s%
}{q}}\left( \alpha p+1\right) \left( \alpha +2\right) ^{1/p}\left(
s+1\right) ^{1/q}\Gamma \left( \alpha +1\right) }  \notag \\
&&\times \left[ \left( \left\vert f^{\prime }\left( a\right) \right\vert
^{q}\left( 2^{s+1}-1\right) +\left\vert f^{\prime }\left( b\right)
\right\vert ^{q}\right) ^{1/q}+\left( \left\vert f^{\prime }\left( a\right)
\right\vert ^{q}+\left\vert f^{\prime }\left( b\right) \right\vert
^{q}\left( 2^{s+1}-1\right) \right) ^{1/q}\right]  \notag
\end{eqnarray}%
where $\frac{1}{p}+\frac{1}{q}=1.$
\end{theorem}

\begin{proof}
Using Lemma \ref{L1}, H\"{o}lder's inequality and the s-convex of $%
\left\vert f^{\prime }\right\vert ^{q}$ it follows that%
\begin{eqnarray*}
&&\left\vert f\left( \frac{a+b}{2}\right) \left[ J_{\left( \frac{a+b}{2}%
\right) -}^{\alpha }g\left( a\right) +J_{\left( \frac{a+b}{2}\right)
+}^{\alpha }g\left( b\right) \right] -\left[ J_{\left( \frac{a+b}{2}\right)
-}^{\alpha }\left( fg\right) \left( a\right) +J_{\left( \frac{a+b}{2}\right)
+}^{\alpha }\left( fg\right) \left( b\right) \right] \right\vert \\
&\leq &\frac{1}{\Gamma \left( \alpha \right) }\left[ \dint_{a}^{\frac{a+b}{2}%
}\left\vert \dint_{a}^{t}\left( s-a\right) ^{\alpha -1}g\left( s\right)
ds\right\vert dt+\dint_{\frac{a+b}{2}}^{b}\left\vert \dint_{t}^{b}\left(
s-a\right) ^{\alpha -1}g\left( s\right) ds\right\vert \left\vert f^{\prime
}\left( t\right) \right\vert ^{q}dt\right] \\
&\leq &\frac{1}{\Gamma \left( \alpha \right) }\left( \dint_{a}^{\frac{a+b}{2}%
}\left\vert \dint_{a}^{t}\left( s-a\right) ^{\alpha -1}g\left( s\right)
ds\right\vert ^{p}dt\right) ^{1/q}\left( \dint_{a}^{\frac{a+b}{2}}\left\vert
f^{\prime }\left( t\right) \right\vert ^{q}dt\right) ^{1/q} \\
&&+\frac{1}{\Gamma \left( \alpha \right) }\left( \dint_{\frac{a+b}{2}%
}^{b}\left\vert \dint_{t}^{b}\left( s-a\right) ^{\alpha -1}g\left( s\right)
ds\right\vert ^{p}dt\right) ^{1/q}\left( \dint_{\frac{a+b}{2}}^{b}\left\vert
f^{\prime }\left( t\right) \right\vert ^{q}dt\right) ^{1/q} \\
&=&\frac{1}{\Gamma \left( \alpha \right) }\left( \dint_{a}^{\frac{a+b}{2}%
}\left\vert \dint_{a}^{t}\left( s-a\right) ^{\alpha -1}g\left( s\right)
ds\right\vert ^{p}dt\right) ^{1/q} \\
&&\times \left[ \left( \dint_{a}^{\frac{a+b}{2}}\left\vert f^{\prime }\left(
t\right) \right\vert ^{q}dt\right) ^{1/q}+\left( \dint_{\frac{a+b}{2}%
}^{b}\left\vert f^{\prime }\left( t\right) \right\vert ^{q}dt\right) ^{1/q}%
\right] \\
&\leq &\frac{\left( b-a\right) ^{\alpha +1}\left\vert \left\vert
g\right\vert \right\vert _{\left[ a,b\right] ,\infty }}{2^{\alpha +1+\frac{s%
}{q}}\left( \alpha p+1\right) \left( \alpha +2\right) ^{1/p}\left(
s+1\right) ^{1/q}\Gamma \left( \alpha +1\right) } \\
&&\times \left[ \left( \left\vert f^{\prime }\left( a\right) \right\vert
^{q}\left( 2^{s+1}-1\right) +\left\vert f^{\prime }\left( b\right)
\right\vert ^{q}\right) ^{1/q}+\left( \left\vert f^{\prime }\left( a\right)
\right\vert ^{q}+\left\vert f^{\prime }\left( b\right) \right\vert
^{q}\left( 2^{s+1}-1\right) \right) ^{1/q}\right] .
\end{eqnarray*}%
Here we use%
\begin{equation*}
\dint_{a}^{\frac{a+b}{2}}\left\vert \dint_{a}^{t}\left( s-a\right) ^{\alpha
-1}g\left( s\right) ds\right\vert ^{p}dt=\frac{\left( b-a\right) ^{\alpha
p+1}}{2^{\alpha p+1}\left( \alpha p+1\right) \alpha ^{p}}
\end{equation*}%
\begin{equation*}
\dint_{a}^{\frac{a+b}{2}}\left\vert f^{\prime }\left( t\right) \right\vert
^{q}dt\leq \frac{b-a}{2^{s+1}\left( s+1\right) }\left[ \left\vert f^{\prime
}\left( a\right) \right\vert ^{q}\left( 2^{s+1}-1\right) +\left\vert
f^{\prime }\left( b\right) \right\vert ^{q}\right]
\end{equation*}%
\begin{equation*}
\dint_{\frac{a+b}{2}}^{b}\left\vert f\prime \left( t\right) \right\vert
^{q}dt\leq \frac{b-a}{2^{s+1}\left( s+1\right) }\left[ \left\vert f^{\prime
}\left( a\right) \right\vert ^{q}+\left\vert f^{\prime }\left( b\right)
\right\vert ^{q}\left( 2^{s+1}-1\right) \right] .
\end{equation*}
\end{proof}

\begin{remark}
In Theorem 9, if we choose $s=1$ , then $\left( \ref{T81}\right) $ reduces
inequality $\left( \ref{8,0}\right) \;$of Theorem 6.
\end{remark}

\end{document}